\documentclass[12pt,english]{article}

\usepackage[T1]{fontenc}
\usepackage[latin9]{inputenc}
\usepackage[verbose,tmargin=2.5cm,bmargin=2.5cm,lmargin=2.5cm,rmargin=2.5cm]{geometry}
\usepackage{array}
\usepackage{verbatim}
\usepackage{mathtools}
\usepackage{enumitem}
\usepackage{multirow}
\usepackage{amsmath}
\usepackage{amsthm}
\usepackage{amssymb}
\usepackage{stackrel}
\usepackage{graphicx}
\usepackage[numbers]{natbib}
\usepackage{hyperref}
\usepackage{babel}
\usepackage{color}
\usepackage{authblk}

\theoremstyle{plain}
\newtheorem{thm}{Theorem}
\newtheorem{prop}[thm]{Proposition}

\newtheorem{lem}[thm]{Lemma}
\newtheorem{cor}[thm]{Corollary}

\theoremstyle{remark}

\theoremstyle{definition}
\newtheorem{defn}{Definition}

\def \kc 	#1{\mathrm{KC}\left(#1\right)}
\def \gp 	#1{G\left(#1\right)}

\def \P 	{\mathrm{GP}}

\def \O 	#1{E_{O}\left(#1\right)}
\def \I 	#1{E_{I}\left(#1\right)}
\def \S 	#1{E_{S}\left(#1\right)}

\def \XYX	{\mathcal{H}}
\def \lcm   {\mathrm{lcm}}
\def \gcd   {\mathrm{gcd}}

\begin{document}

\title{Characterization of generalized Petersen graphs that are Kronecker covers\footnote{Dedicated to Mark Watkins on the Occasion of his 80th Birthday.}}

\author{Matja\v z Krnc\thanks{ \texttt{matjaz.krnc@upr.si} }\,}

%\affil[1]{Department of Computer Science, University of Salzburg, Austria, \texttt{matjaz.krnc@sbg.ac.at}}

\author{Toma\v z Pisanski\thanks{\texttt{tomaz.pisanski@fmf.uni-lj.si}}}
\affil{FAMNIT, University of Primorska, Slovenia}

\maketitle

\begin{abstract}
The family of generalized Petersen graphs $\gp{n,k}$, introduced
by \citet{coxeter1950self} and named by Mark~\citeauthor{watkins1969theorem}~(\citeyear{watkins1969theorem}),
is a family of cubic graphs formed by connecting the vertices of a
regular polygon to the corresponding vertices of a star polygon. The
Kronecker cover $\kc G$ of a simple undirected graph $G$ is a a special
type of bipartite covering graph of $G$, isomorphic to the direct (tensor)
product of $G$ and $K_{2}$.
We characterize all the members of generalized Petersen graphs that
are Kronecker covers, and describe the structure
of their respective quotients.
We observe that some of such quotients are again generalized
Petersen graphs, and describe all such pairs.
\end{abstract}

\noindent

\section{Introduction }

The \emph{generalized Petersen graphs}, introduced by \citet{coxeter1950self}
and named by \citet{watkins1969theorem}, form a very interesting
family of trivalent graphs that can
be described by only two integer parameters. They include
Hamiltonian and non-Hamiltonian graphs, bipartite and
non-bipartite graphs, vertex-transitive and
non-vertex-transitive graphs, Cayley and non-Cayley graphs,
arc-transitive graphs and non-arc transitive graphs, graphs
of girth $3,4,5,6,7$ or $8$.
Their generalization to $I$-graphs does not
introduce any new vertex-transitive graphs but
it contains also non-connected graphs and has in special
cases unexpected symmetries.

Following the notation of \citet{watkins1969theorem}, for a given
integers $n$ and $k<\frac{n}{2}$, we can define a generalized Petersen graph $\gp{n,k}$
as a graph on vertex-set $\left\{ u_{0},\dots,u_{n-1},v_{0},\dots,v_{n-1}\right\} $.
The edge-set may be
 naturally partitioned into three equal parts (note that all subscripts are assumed modulo $n$):
 the edges $\O{n,k}=\left\{u_{i}u_{i+1}\right\}_{i=0}^{n-1}$ from the \emph{outer rim}, inducing a cycle of length $n$;
the edges $\I{n,k}=\left\{ v_iv_{i+k}\right\}_{i=0}^{n-1}$ from the \emph{inner rims}, inducing $\gcd(n,k)$ cycles of length $\frac{n}{\gcd(n,k)}$;
and the edges $\S{n,k}=\left\{u_{i}v_{i}\right\}_{i=0}^{n-1}$, also called \emph{spokes}, that induce a perfect matching in $\gp{n,k}$.
Hence the edge-set may be defined as
$E\left(\gp{n,k}\right)=\O{n,k} \cup \I{n,k} \cup \S{n,k}$.

%A generalized Petersen graph may be bipartite,
%depending on the parity of its parameters.

%Tomo's first handwriting

%\begin{prop}[\citep{boben2005graphs}]
%\label{thm:bipartite-Petersen}A generalized Petersen graph $\gp{n,k}$
%is bipartite, if and only if $n$ is even and $k$ is odd.
%\end{prop}
%The mentioned family of generalized Petersen graphs is denoted by
%$\P$, and is further partitioned into three parts, depending on the
%parity of $n$ and $k$. In particular, let $\ogp\coloneqq\left\{ \gp{n,k}|n\,\mbox{is odd}\right\} $,
%let $\bgp\coloneqq\left\{ \gp{n,k}|n\,\mbox{is even and }k\,\mbox{is odd}\right\} $,
%and let $\egp$ consist of generalized Petersen graphs %$\gp{n,k}$,
%such that both $n$ and $k$ are even, i.e. $\egp\coloneqq\P\setminus\left(\ogp\cup\bgp\right)$.

Various aspects of the structure of the mentioned family has been
observed. Examples include identifying generalized Petersen graphs
that are Hamiltonian \citep{alspach1983classification} or Cayley
\citep{saravzin1997note,nedela1995generalized}, or finding their
automorphism group \citep{steimle2009isomorphism,petkovvsek2009enumeration,horvat2012isomorphism}.
Also, a related generalization to $I$-graphs has been introduced
in the Foster census
\citep{bouwer1988foster}, and further studied
by \citet{boben2005graphs}.

Theory of covering graphs became one of the most important
and successful tools of algebraic graph theory.
It is a discrete analog of the well known theory of covering spaces
in algebraic topology.
In general, covers depend on the values called voltages assigned to the edges of the graphs.
Only in some cases the covering is determined by the graph itself.
One of such cases is the recently studied
\emph{clone cover} \citep{malnivc2014clone}.
The other, more widely known case is the Kronecker cover.

The \emph{Kronecker cover} $\kc G$ (also called bipartite or canonical
double cover) of a simple undirected graph $G$ is a bipartite covering
graph with twice as many vertices as $G$. Formally, $\kc G$ is defined
as a
tensor product $G\times K_2$, i.e. a
graph on a vertex-set
$V\left(\kc G\right)=\left\{ v',v''\right\} _{v\in V(G)}$,
and an edge-set
$E\left(\kc G\right)=\left\{ u'v'',u''v'\right\} _{uv\in E(G)}$.
Some recent work on Kronecker covers includes \citet{gevay2012kronecker}
and \citet{imrich2008multiple}.

In this paper, we study the family of generalized Petersen graphs in
conjunction with the Kronecker cover operation.
Namely, in the next section we state our main theorem
characterizing all the members of
generalized Petersen graphs that
are Kronecker covers, and describing the structure of
their corresponding quotient graphs.
In Section~\ref{sec:id_of_involut} we focus on
the necessary and sufficient conditions for a generalized Petersen
graph to be a Kronecker cover while in Section~\ref{sec:quotients} we
complement the existence results with the description
of the structure of the  corresponding quotient graphs.
We conclude the paper with some remarks and directions for
a possible future research.

\section{Main result}

In order to state the main result we need to introduce the graph and two 2-parametric families of cubic, connected graphs.

\begin{figure}
\centering
\includegraphics[scale=0.55]{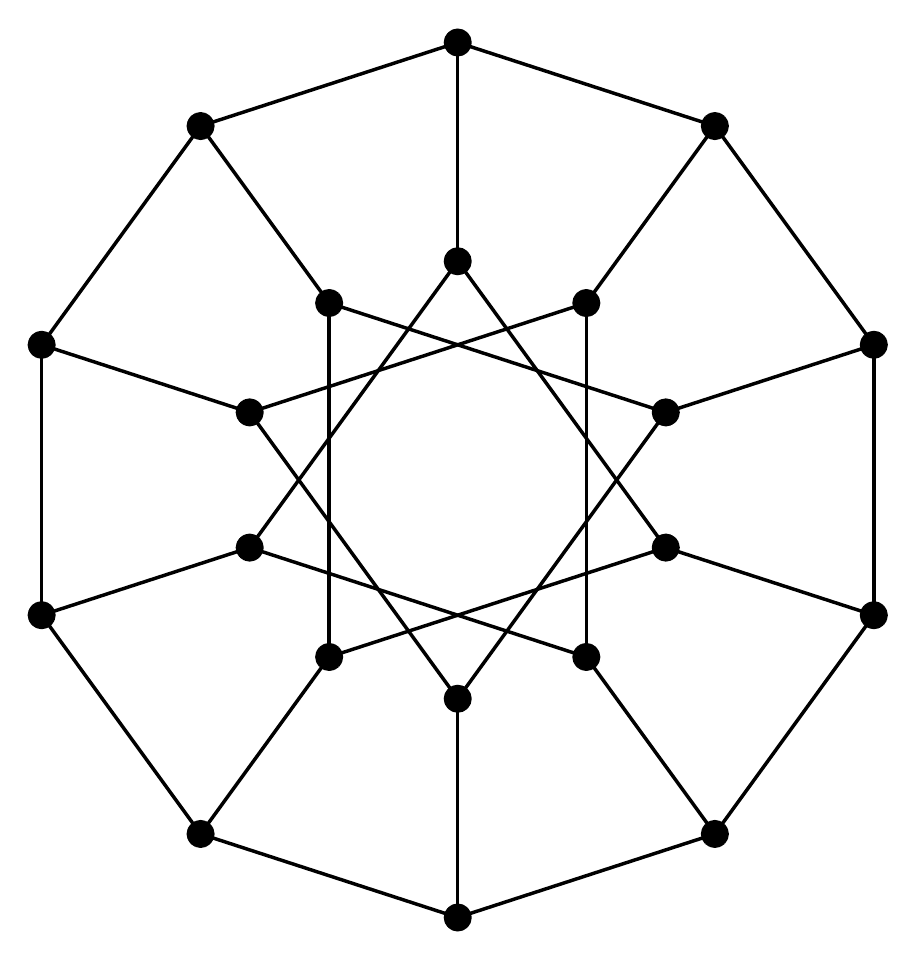}\quad\quad
\includegraphics[scale=0.3]{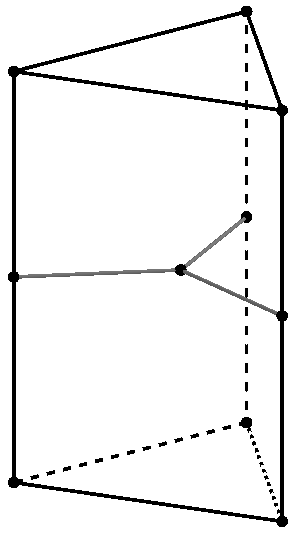}\quad\quad
\includegraphics[scale=0.55]{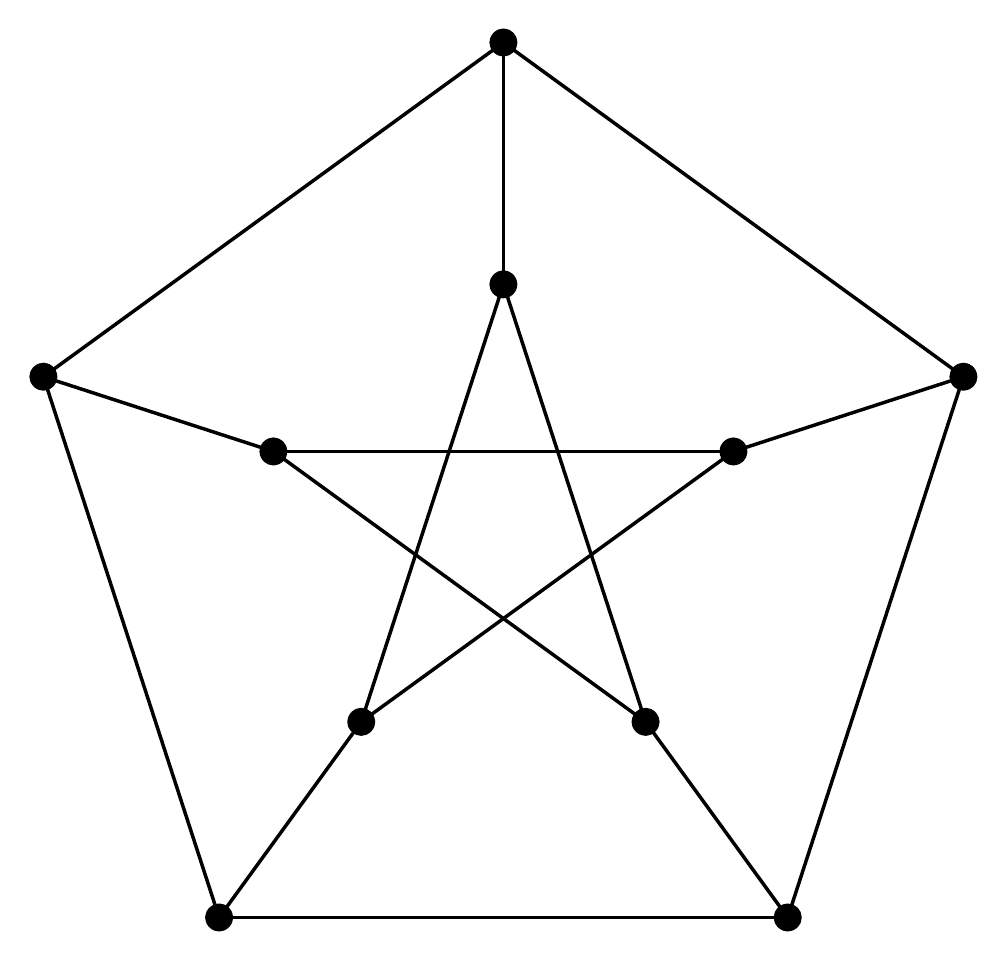}
\caption{The Desargues graph and its both quotients; $\XYX$ and the Petersen graph. \label{fig:gamma}}
\end{figure}

Let $\XYX$  be a graph defined by the following procedure: take the Cartesian product $K_{3}\square P_{3}$,
remove the edges of the middle triangle, add a new vertex in the middle
and connect it to all three $2$-vertices.
Note that the graph $\XYX$ is mentioned in the paper \cite{imrich2008multiple} where it is depicted in Figure 1.

The Desargues graph $\gp{10,3}\simeq \kc{\XYX}$ may be similarly
defined by first taking $K_{6}\square P_{3}$,
removing the edges of the middle hexagon, adding two new vertices in the middle and
alternately connecting them to the consecutuve vertices of the middle hexagon. The mentioned construction of $\XYX$ and the standard drawing of Desargues graph are depicted on Figure~\ref{fig:gamma}.

To describe the quotients of generalized Petersen graphs, we use the LCF notation, named by developers Lederberg, Coxeter and Frucht, for the representation of cubic hamiltonian graphs (for extended
description see \citep{pisanski2012configurations}).

\begin{quote}
In a Hamiltonian cubic graph, the vertices can be arranged in a cycle, which accounts for two
edges per vertex. The third edge from each vertex can then be described by how many positions
clockwise (positive) or counter-clockwise (negative) it leads.
The basic form of the LCF notation
is just the sequence $[a_0,a_1,\dots,a_{n-1}]$
of numbers of positions, starting from an
arbitrarily chosen vertex and written in square brackets.
\end{quote}

To state our results, we only use a special type of such LCF-representable
graphs, namely $C^+(n,k)$ and $C^-(n,k)$, which we define below.

\begin{defn}\label{def:cycle+matching}
Assuming all numbers are modulo $n$, define graphs
$$C^+(n,k)=\left[\frac{n}{2}, \frac{n}{2} + (k-1), \frac{n}{2} + 2(k-1), \dots, \frac{n}{2} + (n-1)(k-1)\right],$$ and similarly

$$C^-(n,k)=\left[\frac{n}{2}, \frac{n}{2} - (k+1), \frac{n}{2} - 2(k+1), \dots, \frac{n}{2} - (n-1)(k+1)\right].$$
\end{defn}

In \cite{imrich2008multiple} it was proven that $\gp{10,3}$ is Kronecker cover of two
non-isomorphic graphs. Here we prove among other things that this is the only
generalized Petersen graph that is
a multiple Kronecker cover. Every other generalized Petersen graph is either a Kronecker cover of a single
graph or it is not a Kronecker cover at all. More precisely;

\begin{thm}\label{thm1} Among the members of the family of
generalized Petersen graphs, $\gp{10,3}$ is the only graph that is the Kronecker
cover of two non-isomorphic graphs, the Petersen graph and the graph $\XYX$.
%Every other $\gp{n,k}$ is either not a Kronecker cover of any graph
%or it is a Kronecker cover of exactly one graph, in particular:
For any other $G\simeq \gp{n,k}$,
the following holds:
\begin{enumerate}[label=$\alph*)$]
\item \label{thm1:non-vt}If $n\equiv2\pmod4$ and $k$ is odd, $G$ is a Kronecker
cover. In particular
	\begin{enumerate}[label=$a_\arabic*)$]
	    \item if $4k<n$, the corresponding quotient graph is $\gp{\frac{n}{2},k}$, and \label{thm1:a1}
	    \item if $n<4k<2n$ the quotient graph is $\gp{\frac{n}{2},\frac{n}{2}-k}$.
	\end{enumerate}
\item \label{thm1:vt}If $n\equiv0\pmod4$ and $k$ is odd, $G$ is a Kronecker cover if and
only if $n\mid \frac{k^{2}-1}{2}$ and  $k< \frac{n}{2}$,
or if $(n,k)=(8,3).$
Moreover,
	\begin{enumerate}[label=$b_\arabic*)$]
	\item if $k=4t+1$ the corresponding quotient is $C^{+}(n,k)$ while \label{thm1:b1}
	\item if $k=4t+3$ the quotient is $C^{-}(n,k)$.\label{thm1:b2}
	\end{enumerate}
\item Any other generalized Petersen graph is not a Kronecker cover.
\end{enumerate}
\end{thm}

For $k = 1$ and even $n$ each $G(n,1)$ is a Kronecker cover. However, if $n = 4t$ case \ref{thm1:b1} applies and the quotient graph is even Moebius ladder.
For $G(4,1)$ the quotient is $K_4 = M_4$.  Similarly, the 8-sided prism $G(8,1)$ is a Kronecker cover of $M_8$.
In case $n = 4t+2$ the case \ref{thm1:a1} applies and the quotient is $G(n/2,1)$. For instance, the 6-sided prism is a Kronecker cover of a 3-sided prism.
For $k>1$ the smallest cases stated in Theorem~\ref{thm1} are presented in Table~\ref{tab:small-cases}.

It is well-known that any automorphism of a connected
bipartite graph either preserves the
two sets of bipartition or interchanges the two
sets of bipartition. In the former case we call the automorphism \emph{colour preserving} and in the latter case \emph{colour reversing}.
Clearly, the product of two color-reversing automorphisms is a color preserving automorphism and the collection of color preserving automorphisms
determine a subgroup of the full automorphism group that is of index 2.

\section{Identifying the Kronecker involutions}\label{sec:id_of_involut}

Before we state an important condition that classifies
Kronecker covers we give the following
definition.

\begin{defn}
A fixed-point free involution $\omega$ that is a color-reversing automorphism of a bipartite graph
is called a \emph{Kronecker involution}.
\end{defn}

We proceed by a well-known proposition by \cite{imrich2008multiple},
regarding the existence of Kronecker covers.
\begin{thm}
\label{prop:involutions}For a bipartite graph $G$, there exists
$G'$ such that $\kc{G'}\simeq G$, if and only if $Aut(G)$ admits
a Kronecker involution. Furthermore,
the corresponding quotient graph may be obtained by contracting all
pairs of vertices, naturally coupled by a given Kronecker involution.
\end{thm}

The following result is well-known. One can find it, for instance in \cite{horvat2012isomorphism}.

\begin{thm}
A generalized Petersen graph $G(n,k)$ is bipartite if
and only if $n$ is even and $k$ is odd.
\end{thm}

We also include the classification concerning symmetry of generalized Petersen graphs,
which follows from the work of Frucht et al.,
\cite{frucht1971groups}
Nedela and \v{S}koviera,
\cite{nedela1995generalized}
 and Lovre\v{c}i\v{c}-Sara\v{z}in \cite{saravzin1997note}.

\begin{thm}[\cite{frucht1971groups,nedela1995generalized,saravzin1997note}]
Let $\gp{n,k}$ be a generalized Petersen graph. Then:
\begin{enumerate}[label=\emph{\alph*)}]
    \item it is symmetric if and only if $\left(n,k\right)\in\left\{ \left(4,1\right),\left(5,2\right),\left(8,3\right),\left(10,2\right),\left(10,3\right),\left(12,5\right),\left(24,5\right)\right\} ,$
    \item it is vertex-transitive if and only if  $k^{2}\equiv \pm 1 \pmod n$ or if $n=10$ and $k=2$,
    \item it is a Cayley graph if and only if $k^{2}\equiv1\pmod n$.
\end{enumerate}
\end{thm}

In general the word \emph{symmetric} means arc-transitive. For cubic graphs this is equivalent to saying vertex-transitive and edge-transitive. For generalized Petersen graph symmetric is equivalent to edge-transitive.

In order to understand which generalized Petersen graphs are Kronecker covers we have to identify all Kronecker involutions for each $G(n,k)$.
In what follows, for a given pair $(n,k)$, our arguments rely on
the structure of the automorphism group $A(n,k)$ of $\gp{n,k}$.
%A statement by \cite{frucht1971groups} can be stated as follows.
This means we have to understand the automorphisms of $G(n,k).$
We define three permutations that may be
defined on the vertex set of a generalized Petersen graph and play an important role in describing its automorphism group.

\begin{defn}
For $i\in\left[0,n-1\right]$, define the permutations $\alpha,\beta$
and $\gamma$ on $V\left(\gp{n,k}\right)$ by
\begin{eqnarray*}
\alpha\left(u_{i}\right)=u_{i+1}, & \quad & \alpha\left(v_{i}\right)=v_{i+1},\\
\beta\left(u_{i}\right)=u_{-i}, & \quad & \beta\left(v_{i}\right)=v_{-i},\\
\gamma\left(u_{i}\right)=v_{ki}, & \quad & \gamma\left(v_{i}\right)=u_{ki}.
\end{eqnarray*}
\end{defn}

\begin{table}
\centering
\begin{tabular}{|l|l|l|l|l|} \hline
$n$ & $k$ & case & involution & quotient \\ \hline \hline
$4$ & $1$ & $b_1$ & $ \alpha^{2} \gamma$ & $C^+(4,1)$ \\ \hline
$6$ & $1$ & $a_1$ & $\alpha ^{3}$ & $G(3,1)$ \\ \hline
$8$ & $1$ & $b_1$ & $ \alpha^{4} \gamma$ & $C^+(8,1)$ \\ \hline
$8$ & $3$ & $b_2$ & $ \alpha^{4} \beta \gamma$ & $C^-(8,3)$ \\ \hline
$10$ & $1$ & $a_1$ & $\alpha ^{5}$ & $G(5,1)$ \\ \hline
$10$ & $3$ & $\cdot$ & $ \alpha^5,\Delta$ & $G(5,2),X$ \\ \hline
$12$ & $1$ & $b_1$ & $ \alpha^{6} \gamma$ & $C^+(12,1)$ \\ \hline
$12$ & $5$ & $b_1$ & $ \alpha^{6} \gamma$ & $C^+(12,5)$ \\ \hline
$14$ & $1$ & $a_1$ & $\alpha ^{7}$ & $G(7,1)$ \\ \hline
$14$ & $3$ & $a_1$ & $\alpha ^{7}$ & $G(7,3)$ \\ \hline
$16$ & $1$ & $b_1$ & $ \alpha^{8} \gamma$ & $C^+(16,1)$ \\ \hline
$18$ & $1$ & $a_1$ & $\alpha ^{9}$ & $G(9,1)$ \\ \hline
$18$ & $3$ & $a_1$ & $\alpha ^{9}$ & $G(9,3)$ \\ \hline
$18$ & $5$ & $a_2$ & $\alpha ^{9}$ & $G(9,4)$ \\ \hline
$20$ & $1$ & $b_1$ & $ \alpha^{10} \gamma$ & $C^+(20,1)$ \\ \hline
$20$ & $9$ & $b_1$ & $ \alpha^{10} \gamma$ & $C^+(20,9)$ \\ \hline
$22$ & $1$ & $a_1$ & $\alpha ^{11}$ & $G(11,1)$ \\ \hline
$22$ & $3$ & $a_1$ & $\alpha ^{11}$ & $G(11,3)$ \\ \hline
$22$ & $5$ & $a_1$ & $\alpha ^{11}$ & $G(11,5)$ \\ \hline
$24$ & $1$ & $b_1$ & $ \alpha^{12} \gamma$ & $C^+(24,1)$ \\ \hline
$24$ & $7$ & $b_2$ & $ \alpha^{12} \beta \gamma$ & $C^-(24,7)$ \\ \hline
$26$ & $1$ & $a_1$ & $\alpha ^{13}$ & $G(13,1)$ \\ \hline
$26$ & $3$ & $a_1$ & $\alpha ^{13}$ & $G(13,3)$ \\ \hline
$26$ & $5$ & $a_1$ & $\alpha ^{13}$ & $G(13,5)$ \\ \hline
$26$ & $7$ & $a_2$ & $\alpha ^{13}$ & $G(13,6)$ \\ \hline
\end{tabular}
\caption{\label{tab:small-cases}The smallest generalized Petersen graphs that are Kronecker covers, together with their corresponding  Kronecker involutions $\omega$ and the quotient graphs.}
\end{table}

Let us paraphrase Theorem 5 of Lover\v{c}i\v{c}-Sara\v{z}in \cite{saravzin1997note} that follows from Frucht et al. \cite{frucht1971groups}

\begin{thm}\label{thm:aut-groups}

If $(n,k)$ is not one of $(4,1),(5,2),(8,3),(10,2),(10,3),(12,5)$, or $(24,5)$, then the following holds:
\begin{itemize}
\item if $k^2 \equiv 1$ mod $n$, then
$$A(n,k) = \langle \alpha, \beta, \gamma | \alpha^n = \beta^2 = \gamma^2 = 1, \alpha\beta = \beta\alpha^{-1}, \alpha\gamma = \gamma\alpha^k, \beta\gamma = \gamma\beta \rangle$$
\item if $k^2 \equiv -1$ mod $n$, then
$$A(n,k) = \langle \alpha, \beta, \gamma | \alpha^n = \beta^2 = \gamma^4 = 1, \alpha\beta = \beta\alpha^{-1}, \alpha\gamma = \gamma\alpha^k, \beta\gamma = \gamma\beta  \rangle$$
In this case $\beta = \gamma^2$.
\item In all other cases the graph $G(n,k)$ is not vertex-transitive and
$$A(n,k) = \langle \alpha, \beta | \alpha^n = \beta^2 = 1, \alpha\beta = \beta\alpha^{-1}  \rangle$$

\end{itemize}
\end{thm}

Observe that in the case when the underlying graph is symmetric, the automorphism group may not be described by $\alpha, \beta$ and $\gamma$.
For illustration, consider the following permutation $\Delta$ of the vertex-set of Desargues graph $\gp{10,3}$, as constructed on Figure~\ref{fig:gamma}, where all the vertices from $C_6\square P_3$ are rotated around hexagon for $180^{\circ}$, while the additional two vertices are swapped.
Since any member of $\left < \alpha,\beta,\gamma\right >$
 either fixes each rim set-wise or swaps them,
it is easy to see that $\Delta\in Aut(\gp{10,3})$ is not generated by $\alpha, \beta$ nor $\gamma$. Hence the symmetric bipartite members of $\P$ need to be checked separately.
It turns out that all quotients of $\P$ may be obtained
by Kronecker involutions from $\left < \alpha,\beta,\gamma\right >$,
or by $\Delta$ (in the case $n=10$ and $k=3$).

However, for the non-symmetric members of generalized Petersen graphs, Theorem~\ref{thm:aut-groups} implies that any element of automorphism group (including any Kronecker involution) may be
expressed in terms
of $\alpha, \beta$ and $\gamma$. In fact, in the next lemma we show that any such element
may be expressed in a canonical way.

\begin{lem}
	\label{rem:autom-shuffle}
%Let $\gp{n,k}$ be a non-symmetric generalized Petersen graph.
For any automorphism $\omega$ from $A(n,k)$ we may associate a unique triple $(a,b,c) \in \mathbb{Z}_n \times \mathbb{Z}_2 \times \mathbb{Z}_2$
	such that $\omega=\alpha^{a}\beta^{b}\gamma^{c}$,
    whenever $(n,k)$ is not equal to one of the symmetric pairs $(4,1),(5,2),(8,3),(10,2),(10,3),(12,5)$, or $(24,5)$.
\end{lem}

\begin{proof}
Let $\gp{n,k}$ be a generalized Petersen graph and let $a,b,c$ be arbitrary integers. Then, by definition of the three generators $\alpha,\beta,\gamma$ (or the three permutations) it clearly holds
\begin{enumerate}
    \item $\beta\alpha^{a}=\alpha^{-a}\beta$, \label{enu:2}
    \item $\gamma\alpha^{a}=\alpha^{ak}\gamma$.  If $k^2 \equiv 1$ mod $n$.\label{enu:1}
    \item $\gamma\alpha^{a}=\alpha^{-ak}\gamma$.  If $k^2 \equiv -1$ mod $n$.\label{enu:4}
    \item $\gamma\beta=\beta\gamma$,\label{enu:3}
\end{enumerate}
We omit the arguments for (\ref{enu:2}) and (\ref{enu:3}) as they are repeated from the definition.
Property (\ref{enu:1}) follows from the facts
 $\alpha\gamma = \gamma\alpha^k$ and $k^2 \equiv 1$ mod $n$. Since  $\alpha^a\gamma = \gamma\alpha^{ak}$ for any $a$,
 take $a = k$ and we get $\alpha^k\gamma = \gamma\alpha^{k^2} = \gamma\alpha$ and the result follows. In a similar way we prove
 (\ref{enu:4}).

 By using the commuting rules (\ref{enu:2}--\ref{enu:3}) above we may transform any product of permutations $\alpha,\beta,\gamma$ to a form $\alpha^a\beta^b\gamma^c$
 with $0\leq b,c\leq 1$.
 In non-vertex-transitive case we have $c = 0$ while in vertex-transitive non-Cayley case, one could have $\gamma, \gamma^2, \gamma^3$. However, we may always
 use the fact that $\gamma^2 = \beta$ and the result follows readily.
\end{proof}

Note that in a bipartite $G(n,k)$ automorphisms $\alpha$ and $\gamma$ are color reversing, while while $\beta$ is color preserving.

\begin{prop}\label{prop:dihedral+}
The following statements hold:
\begin{enumerate}
    \item \label{enu:Dn1}$\alpha^a$ is a Kronecker involution iff. $a=n/2$ and $n\equiv 2 \pmod 4$;
    \item \label{enu:Dn2}$\alpha^a\beta$ is not a Kronecker involution;
    \item \label{enu:Dn3}if $k^2\equiv-1\pmod n$, then neither $\alpha^a\gamma$ nor $\alpha^a\beta\gamma$ is a Kronecker involution, for any admissible $a$.

\end{enumerate}

\end{prop}
\begin{proof} We prove the claims separately.

(\ref{enu:Dn1})
Let $\omega=\alpha^a$ be a Kronecker involution. It is clear that $\omega$ does not fix any edge, and since $\omega$ is involution we trivially have $a=\frac{n}{2}$.
But since $\omega$ is must be color-reversing, $a$ must at the same time be odd, hence the conclusion.

(\ref{enu:Dn2})
Let $\omega=\alpha^a\beta$ be a Kronecker involution. Since $\omega$ is color-reversing, $a$ must be odd. Letting $i=\frac{a-1}{2}$, it is enough to observe that an edge $u_iu_{i+1}$ is fixed by $\omega$.

(\ref{enu:Dn3})
In both cases, the resulting squared permutation can be written in form $\alpha^{a'}\beta$,
which contradicts the fact that the original permutation is an involution.

\end{proof}

In every generalized Petersen graph $G(n,k)$ permutations $\alpha$ and $\beta$ are automorphisms. Moreover, they generate
the dihedral group $D_n$ of order $2n$ of automorphisms which is, in general, a subgroup of the full automorphism group $A(n,k)$.
The two vertex orbits under $D_n$ are exactly the outer rim and the inner rim and the three edge orbits are outer-rim, inner-rim and
the spokes. Clearly, Proposition~\ref{prop:dihedral+} deals with Kronecker involutions
from $D_n$ and in particular implies the condition for $\gp{n,k}$ being Kronecker cover
described in \ref{thm1:non-vt} of Theorem~\ref{thm1}. But additional Kronecker involutions
may exist by the fact that the automorphism group of a generalized Petersen graph may be larger
then $D_n$. In the next subsection we describe these additional Kronecker involutions, which may (see  (\ref{enu:Dn3}) of Proposition~\ref{prop:dihedral+}) only happen with $k^2\equiv 1 \pmod n$.

\subsection{Additional Kronecker involutions with $k^2\equiv 1 \pmod n$}
In what follows, we assume $k^2 \equiv 1 \pmod n$ and define $Q$, such that $k^2-1= Qn$. The only two permissible types of involutions are $\alpha^a\gamma$ and $\alpha^a\beta\gamma$.

For an integer $i$ let $b(i)$ correspond to the uniquely defined maximal integer, such that $2^{b(i)}$ divides $i$. In particular, we have
\begin{align}\label{b(n)}
b(n)=b(k+1)+b(k-1)-b(Q).
\end{align}

In the following two subsections, we the condition for a generalized Petersen graph being a Kronecker cover, described in (\ref{thm1:b1} and (\ref{thm1:b2} of Theorem~\ref{thm1}, respectively.

\subsubsection*{Involutions of type $\alpha^a\gamma$}

We have $\omega_a=\alpha^a\gamma$ such that $\omega_a(v_i)=u_{ki+a}$ and
$\omega_a(u_i)=v_{ki+a}$, so let us for easier notation define a function
$\Omega_a:\mathbb{Z}_n\rightarrow \mathbb{Z}_n$ such that $\Omega_a(i)=ki+a$.
By these definitions, we clearly have the following properties:
\begin{enumerate}[label=$\mathrm P\arabic*.$]
    \item Permutation $\omega_a$ is color reversing if and only if $\Omega_a(i)\equiv i\pmod 2$, in other words if $a$ is even and $k$ is odd. \label{prop1:even-a}
    \item Permutation $\omega_a$ is an involution if and only if $\Omega_a(\Omega_a(i)) \equiv i \pmod n$, i.e. if $ak+a\equiv 0 \pmod n$.\label{prop2}
    \item Permutation $\omega_a$ may only fix a spoke. In particular, $\alpha_a$ fixes some edge if and only if there exists an integer $i$, such that
    $\Omega_a(i)\equiv i\pmod n$.
\end{enumerate}
Finally, let us define a constant $a_{\min}=\frac{n}{\gcd(n,k+1)}$.
The next lemma describes necessary conditions for $\omega_a$ to be a Kronecker involution.
\begin{lem}\label{nec-conditions}
Let $\omega_a$ be a Kronecker involution. Then the following claims are true:
\begin{enumerate}[label=$(C_\arabic*)$]
    \item \label{2a}$\omega_{2a}$ is not a Kronecker involution;
    \item \label{amin}there exists an odd integer $s$, such that  $a=sa_{\min}$;
    \item \label{amin-even} $a_{\min}$ is even;
    \item \label{oddQ} $Q$ is even;
    \item \label{k=1mod4} $k\equiv 1 \pmod 4$.

\end{enumerate}
\end{lem}

\begin{proof} We prove the claims consecutively.

\ref{2a}
Since $\omega_a$ is an involution, by (\ref{prop2}) we have $a+ka\equiv0\pmod n$ and
$$
\Omega_{2a}(a)\equiv (ak+a)+a\equiv a \pmod n,
$$
hence $a$ is a fixed point.

\ref{amin}
From (\ref{prop2}) it follows that $a(k+1)$ is a multiple of $n$. In other words, there exists
a positive integer $C$, such that $a=\frac{Cn}{k+1}$. It is clear that $a$ is minimized whenever $Cn=\lcm(k+1,n)$, i.e.
$$
a_{\min}=\frac{\lcm(k+1,n)}{k+1}=\frac{n}{\gcd(n,k+1)}.
$$
Note that in general $C$ may be some $s$-th multiple of $\lcm(k+1,n)$, however the value of $s$ from our claim may by \ref{2a} only be odd.

\ref{amin-even} The claim follows from (\ref{prop1:even-a}) and \ref{amin} by the fact that $a$ is
even if and only if $a_{\min}$ is even.

\ref{oddQ}
Suppose that $Q$ is odd, which implies that $b(Q)=0$ and $b(n)> b(k+1)$.
By \ref{amin} we have that $a=sa_{\min}$, for some odd $s$.
We will show that $\Omega_a$ have a fixed point in
$r=\frac{k+a+1}{2}$. Indeed, we have
\begin{align*}
\Omega_a(r)-r&\equiv \left ( \frac{k+1}{2}+\frac{a}{2}\right )(k-1)+a\\
& \equiv \frac{Qn}{2}+\frac{k+1}{2}a\equiv
\frac{n}{2}+\frac{s\cdot \lcm(n,k+1)}{2}\pmod n.
\end{align*}
To conclude the proof it is enough to notice that $\lcm(n,k+1)$ is an odd multiple of $n$ by the fact that $b(n)> b(k+1)$.

\ref{k=1mod4}
By contradiction assume that $k+1 \equiv 0 \pmod 4$. We will show that in such case
we have $b(n)=b(\gcd(n,k+1))$, implying that $a_{\min}$ is odd. Let $t=b(n)$ and $s=b(k+1)$ and note that by \ref{oddQ} we have
$$
t=s+b(k-1)-b(Q)\leq s+b(k-1)-1=s,
$$
hence there exist odd integers $o_1,o_2$ such that $n=o_1 2^t$ and $k+1=o_2 2^s$, and
\begin{align*}
a_{\min}&= \frac{n}{\gcd(n,k+1)}
=\frac{o_12^t}{\gcd(o_1 2^t,o_2 2^s)}\\
&=\frac{o_1}{\gcd(o_1,o_2 2^{s-t})}\equiv 1\pmod 2.
\end{align*}
As desired, we are in contradiction with \ref{amin-even}.
\end{proof}

We conclude the proof of condition (\ref{thm1:b1} of Theorem~\ref{thm1} by showing that the necessary conditions \ref{oddQ} and \ref{k=1mod4} mentioned in Lemma~\ref{nec-conditions} are in fact sufficient.

\begin{prop}\label{prop:sufficient}
Let $Q$ be even and $k\equiv 1 \pmod 4$. Then $\omega_{a}$ is a Kronecker involution for any $a=sa_{\min}$ and odd $s$.
\end{prop}

\begin{proof}
We consequently prove that $\omega_a$ is an involution, that it is color-reversing,
and that it does not fix any edge.
Observe that $\omega_a$ is an involution by the fact that
$$
(k+1)a\equiv (k+1)s\cdot \frac{\lcm(n,k+1)}{k+1}\equiv 0 \pmod n.
$$
Notice that  $k \equiv 1 \pmod 4$ implies $\gcd(n,k+1) \equiv 2 \pmod 4$
hence $a_{\min}$ and $a$ are even which implies that $\omega_a$ is color-reversing.
Using $(k+1)(k-1)=Qn$, it follows
\begin{align}
a_{\min}=\frac{Qn}{Q\gcd(n,k+1)}=
\frac{(k+1)(k-1)}{(k+1)\cdot\gcd(k-1,Q)}=\frac{k-1}{\gcd(k-1,Q)}.\label{eq:alt-notion-amin}
\end{align}
Now suppose that $\omega_a$ fixes a spoke $u_iv_i$. Then
\begin{align*}
0 &\equiv i(k-1)+a \pmod n\\
  &\equiv i\cdot a_{\min}\cdot \gcd(k-1,Q)+s\cdot a_{\min} \pmod n\\
  &\equiv a_{\min}(i\cdot \gcd(k-1,Q) +s) \pmod n,
\end{align*}
which is equivalent to
\begin{align}
i\cdot \gcd(k-1,Q) +s \equiv 0 \pmod {\gcd(n,k+1)}\label{eq:contr:1}
\end{align}
by the fact that $a_{\min}\mid n$. But (\ref{eq:contr:1}) is a contradiction since clearly both $\gcd(n,k+1)$ and $\gcd(Q,k-1)$ are even while $s$ is odd.
\end{proof}

\subsubsection*{Involutions of type $\alpha^a\beta\gamma$}

In this section we focus on Kronecker involutions  that also include reflection $\beta$.
While this fact requires some adjustments by the fact that we are now
considering involutions of type $\alpha^a\beta\gamma$, the subsection is
mostly a compact transcript of the previous one.

Define $\omega'_a=\alpha^a\beta\gamma$ i.e. $\omega'_a(v_i)=u_{a-ki}$ and
$\omega'_a(u_i)=v_{a-ki}$, and let
$\Omega'_a:\mathbb{Z}_n\rightarrow \mathbb{Z}_n$ be a function defined as $\Omega'_a(i)=a-ki$.
In this case, the requirements for the $\omega'_a$ being
Kronecker involution imply:
\begin{enumerate}
    \item Permutation $\omega'_a$ is color reversing if and only if $a$ is even and $k$ is odd.
    \item Permutation $\omega'_a$ is an involution if and only if $a-ak\equiv 0 \pmod n$.
    \item Permutation $\omega'_a$ may only fix an $i$-th spoke if and only if there exists an integer $i$, such that $\Omega'_a(i)\equiv i\pmod n$.
\end{enumerate}
Define a constant $a'_{\min}=\frac{n}{\gcd(n,n-k+1)}$, and note the following necessary conditions for $\omega'_a$ to be a Kronecker involution.
\begin{lem}\label{_nec-conditions}
Let $\omega'_a$ be a Kronecker involution. Then the following claims are true:
\begin{enumerate}[label=$(C'_\arabic*)$]
    \item \label{_2a}$\omega_{2a}$ is not a Kronecker involution;
    \item \label{_amin}there exists an odd integer $s$, such that  $a=sa_{\min}$;
    \item \label{_amin-even} $a_{\min}$ is even;
    \item \label{_oddQ} $Q$ is even;
    \item \label{_k=1mod4} $k\equiv -1 \pmod 4$.
\end{enumerate}
\end{lem}

\begin{proof} We omit the proofs of \ref{_2a}, \ref{_amin-even} and \ref{_k=1mod4}, since they
may be transcribed from the proofs of \ref{2a}, \ref{amin-even} and \ref{k=1mod4} along the same lines.

%(\ref{_a-even}) follows from the fact that $k$ is odd and $\omega'_a$ is color-reversing if and only if for each $r$ we have $a-kr\equiv r\pmod 2$.

%(\ref{_2a})
%If $\omega'_a$ is an involution, then $ka-a\equiv 0 \pmod n$. Now take
%$$
%\Omega'_{2a}(a)\equiv (a-ak)+a\equiv a \pmod n,
%$$
%hence $a$ is a fixed point.

\ref{_amin}
Since $a(n-k+1)$ is a multiple of $n$, there exists a constant $C$ such that $a=\frac{Cn}{n-k+1}$. It is clear that $a$ is minimized whenever $Cn=\lcm(n-k+1,n)$, i.e.
$$
a_{\min}=\frac{\lcm(n-k+1,n)}{n-k+1}=\frac{n}{\gcd(n,n-k+1)}.
$$
Again, $C$ may be some multiple of $\lcm(k+1,n)$, however the value of $s$ from our claim may by \ref{_2a} only be odd.

%(\ref{_amin-even}) follows from (\ref{_a-even}) and (\ref{_amin}) by the fact that $a$ is
%even if and only if $a_{\min}$ is even.

\ref{_oddQ}
Suppose that $Q$ is odd and let $\omega_a$ be such a Kronecker involution.
By (\ref{_amin}) we have that $a=sa'_{\min}$, for some odd $s$. Note that $b(Q)=0$
implies $b(k-1) < b(n)$ and
$$
b(n-(k-1))=\min\left (b(n),b(k-1)\right )=b(k-1),
$$
hence $b(n)> b(n-k+1)$. We will show that $\Omega'_a$ have a fixed point in
$r=\frac{k-1+a}{2}$. Indeed, we have
\begin{align*}
\Omega'_a(r)-r&\equiv a-\left ( \frac{k-1}{2}+\frac{a}{2}\right )(k+1)
\equiv a-\frac{Qn}{2}-\frac{a}{2}\cdot (k+1) \pmod n\\
& \equiv -\frac{Qn}{2}+\frac{a}{2}\cdot (n-k+1)\equiv
\frac{n}{2}-\frac{s\cdot \lcm(n,n-k+1)\cdot(n-k+1)}{2\cdot (n-k+1)}\pmod n \\
&\equiv \frac{n}{2}-\frac{s\cdot \lcm(n,n-k+1)}{2} \pmod n.
\end{align*}
To conclude the proof it is enough to notice that $\lcm(n,n-k+1)$ is an odd multiple of $n$ by the fact that $b(n)\geq b(n-k+1)$.

%\ref{_k=1mod4}
%By contradiction assume that $k-1 \equiv 0 \pmod 4$, and let $t=b(n)$ and $s=b(k-1)=b(n-k+1)$. Clearly we have
%$$
%t\leq s+b(k+1)-1=s,
%$$
%hence there exist odd integers $o_1,o_2$ such that
%\begin{align*}
%a_{\min}&= \frac{n}{\gcd(n,n-k+1)}
%=\frac{o_12^t}{\gcd(o_1 2^t,o_2 2^s)}\\
%&=\frac{o_1}{\gcd(o_1,o_2 2^{s-t})}\equiv 1\pmod 2.
%\end{align*}
%But now we are in contradiction with (\ref{_amin-even}).
\end{proof}

The proposition below shows that the necessary conditions mentioned in Lemma~\ref{_nec-conditions} are in fact sufficient, which
proves the condition (\ref{thm1:b2} of Theorem~\ref{thm1}.

\begin{prop}\label{_prop:sufficient}
Let $Q$ be even and $k\equiv 3 \pmod 4$. Then $\omega'_{a}$ is a Kronecker involution for any $a=sa'_{\min}$ and odd $s$.
\end{prop}

\begin{proof}
We consequently prove that $\omega'_a$ is an involution, that it is color-reversing,
and that it does not fix any edge.
Observe that $\omega'_a$ is an involution by the fact that
$$
(1-k)a\equiv (n+1-k)s\cdot \frac{\lcm(n-k+1,n)}{n-k+1}\equiv 0 \pmod n.
$$
Notice that  $k \equiv -1 \pmod 4$ implies $\gcd(n,n-k+1) \equiv 2 \pmod 4$
hence $a_{\min}$ and $a$ are even which implies that $\omega_a$ is color-reversing.
Using $(k+1)(k-1)=Qn$, it follows
\begin{align}
a'_{\min}&=\frac{Qn}{Q\cdot \gcd(n,n-k+1)}
=\frac{(k+1)(k-1)}{\gcd\left (Qn,(k+1-Q)(k-1)\right )}
\nonumber \\
&=\frac{(k+1)(k-1)}{(k-1)\cdot\gcd(k+1,k+1-Q)}
=\frac{k+1}{\gcd(k+1,k+1-Q)}.
\label{_eq:alt-notion-amin}
\end{align}
Now suppose that $\omega'_a$ fixes a spoke $u_iv_i$ and, for easier notation,
let $g=\gcd(k+1,k+1-Q)$. Then
\begin{align*}
0 &\equiv a-i(k+1) \pmod n\\
  &\equiv s\cdot a'_{\min}-i\cdot a'_{\min}\cdot g  \pmod n\\
  &\equiv a'_{\min}(s-i\cdot g) \pmod n,
\end{align*}
which is equivalent to
\begin{align}
s-i\cdot g \equiv 0 \pmod {\gcd(n,n-k+1)}\label{_eq:contr:1}
\end{align}
by the fact that $a'_{\min}\mid n$. But (\ref{_eq:contr:1}) is a contradiction since clearly both $\gcd(n,n-k+1)$ and $\gcd(k+1,k+1-Q)$ are even while $s$ is odd.
\end{proof}

We conclude with an important corollary that holds for Kronecker involutions of both types.

\begin{cor}
If $Q$ is even
then $\omega_{n/2}$ is a Kronecker involution.
\end{cor}
\begin{proof}
Suppose that $k\equiv 1 \pmod 4$ and let $s=\frac{\gcd(n,k+1)}{2}$. Since $k+1\equiv 2 \pmod 4$, clearly also $\gcd(n,k+1)\equiv 2 \pmod 4$ and $s$ is an odd integer. But then $\omega_{s\cdot a_{\min}}$ is a Kronecker involution by Proposition~\ref{prop:sufficient}, and clearly $s\cdot a_{\min}=n/2$.

Similarly, if $k\equiv -1 \pmod 4$, set $s'=\frac{\gcd(n,n-k+1)}{2}$ and
notice that it is an odd integer, while $s'\cdot a'_{\min}=n/2$.
\end{proof}

In the next section we prove that for any member of $\P$ except $\gp{10,3}$, all quotients are
isomorphic, hence it will be conveniant to always (when applicable) use the cannonical value of $a=n/2$.

\section{The quotients of generalized Petersen graphs}\label{sec:quotients}
For a given generalized Petersen graph, so far we identified all its Kronecker involutions.
%For clarity, we list these in Table XXYY.
In this section we determine the structure of the corresponding quotient graphs, for each of these involutions.
Namely, the next two subsections deal with the structural part of the
statements \ref{thm1:non-vt} and \ref{thm1:vt} of the
Theorem~\ref{thm1}, respectively.

\subsection{Involutions of $D_n$}

We already know that the only Kronecker involution in the Dihedral group is the rotation
$\alpha^{n/2}$, which is realized whenever $n\equiv 2 \pmod 4$ and $k$ is odd.
In order to prove \ref{thm1:non-vt} of Theorem~\ref{thm1},
it is enough to show the following proposition,
which describes the corresponding quotient graph explicitly.

\begin{prop}
\label{thm:main}For an odd $n$ and an integer $k<\frac{n}{2}$, we have
\[
\kc{\gp{n,k}}\simeq\begin{cases}
\gp{2n,k}; & \quad k\mbox{ is odd};\\
\gp{2n,n-k}; & \quad k\mbox{ is even.}
\end{cases}
\]
\end{prop}
%We start by showing the left implication of Proposition~\ref{thm:main}.

\begin{proof}[Proof of \ref{thm1:non-vt} from Theorem~\ref{thm1}]
Let $G\simeq G(n,k)$ and $G'\simeq\kc G$, for an odd integer $n$
and $k<\frac{n}{2}$.
The edges of $\kc G$ are naturally partitioned to the following three
groups:
\begin{enumerate}[label=(E\arabic*),ref=(E\arabic*)]
\item $u'_{i}v''_{i}$ and $u''_{i}v'_{i}$;\label{enu:spikes}
\item $u'_{i}u''_{i+1}$ and $u''_{i}u'_{i+1}$;\label{enu:outer-rim}
\item $v'_{i}v''_{i+k}$ and $v''_{i}v'_{i+k}$.\label{enu:inner-rims}
\end{enumerate}
For easier notation, define $k'$ to be equal
$k$ or $n-k$, depending on whether $k$ is odd or even, respectively.
Furthermore, let $H\coloneqq\gp{2n,k'}$ and denote its vertex set
with
\[
V(H)=\left\{ a_{0},\dots,a_{2n-1},b_{0},\dots,b_{2n-1}\right\} ,
\]
while its edge set is consisted of edges of form $a_{i}a_{i+1}$,
$a_{i}b_{i}$ and $b_{i},b_{i+k'}$. To show the left implication
of Proposition~\ref{thm:main}, it is enough to show that $G'\simeq H$.
Throughout the proof all subscripts for vertices from~$H$ (on the
left-hand side) are assumed to be modulo $2n$, while all subscripts
for vertices from $G'$ (on the right-hand side) are assumed to be
modulo $n$. To show an equivalence, we introduce a bijection $f:V(H)\rightarrow V(G')$,
such that
\begin{eqnarray*}
a_{i} & \mapsto & \begin{cases}
u'_{i} & \quad\mbox{if }i\mbox{ is even,}\\
u''_{i} & \quad\mbox{if }i\mbox{ is odd,}
\end{cases}\quad\mbox{and}\quad b_{i}\mapsto\begin{cases}
v''_{i} & \quad\mbox{if }i\mbox{ is even,}\\
v'_{i} & \quad\mbox{if }i\mbox{ is odd.}
\end{cases}
\end{eqnarray*}
for any $0\leq i<2n$. Since $n$ is odd, $f$ is clearly a bijection
and it is enough to show that $f$ is a homomorphism between $H$
and $G'$. We now check that all edges from $H$ map to edges in $G'$.
First observe that in $H$, edges of types $a_{i}a_{i+1}$ and $a_{i}b_{i}$
map to these in \ref{enu:outer-rim} and \ref{enu:spikes}, respectively.
Indeed, by definition we have
\[
f\left(a_{i}a_{i+1}\right)=\begin{cases}
u'_{i}u''_{i+1} & \quad\mbox{if }i\mbox{ is even,}\\
u''_{i}u'_{i+1} & \quad\mbox{if }i\mbox{ is odd,}
\end{cases}\quad\mbox{and}\quad f\left(a_{i}b_{i}\right)=\begin{cases}
u'_{i}v''_{i} & \quad\mbox{if }i\mbox{ is even,}\\
u''_{i}v'_{i} & \quad\mbox{if }i\mbox{ is odd.}
\end{cases}
\]
Finally, for edges of type $b_{i},b_{i+k'}$, we now observe that
\begin{equation}
f\left(b_{i}b_{i+k'}\right)=\begin{cases}
v''_{i}v'_{i+k'} & \quad\mbox{if }i\mbox{ is even,}\\
v'_{i}v''_{i+k'} & \quad\mbox{if }i\mbox{ is odd}.
\end{cases}\label{eq:inner-edges}
\end{equation}
Indeed, if $k$ is odd or even, we have
\[
b_{i+k}\mapsto\begin{cases}
v''_{i+k} & \quad\mbox{if }i+k\mbox{ is even,}\\
v'_{i+k} & \quad\mbox{if }i+k\mbox{ is odd.}
\end{cases}\quad\mbox{and}\quad b_{i-k+n}\mapsto\begin{cases}
v''_{i+n-k} & \quad\mbox{if }i-k+n\mbox{ is even,}\\
v'_{i+n-k} & \quad\mbox{if }i-k+n\mbox{ is odd,}
\end{cases}
\]
respectively. Keep in mind that all subscripts on the right side are
modulo $n$. By (\ref{eq:inner-edges}) we conclude that edges of
type $b_{i}b_{i+k'}$ correspond to the edges of type \ref{enu:inner-rims}
in $G'$. Since both $G'$ and $H$ are by definition cubic and of
the same cardinality, the isomorphism follows.%
\begin{comment}
First observe that in $f$, the spikes from $H$ map to a perfect
matching in $G'$. Furthermore, since $n$ is odd, the outer-rim of
$H$ maps to $2n$-cycle in $G'$. We fix the sets of edges from \ref{enu:spikes}
and \ref{enu:outer-rim} to be the spikes and the outer-rim of $G'$,
respectively. Also note that, since $n$ is odd, the edges from \ref{enu:outer-rim}
induce a $2n$-cycle that we fix to be the outer-rim. and these from
\ref{enu:spikes} to induce the outer rim and the matching spikes,
respectively. that the outer rim of $G$ maps to $2n$-cycle in $G'$
that we fix to be the outer rim labeled
\[
u'_{0},u''_{1},\dots,u'_{n-1},u''_{0},\dots,u''_{n-1}.
\]
\end{comment}
\end{proof}

It remains to describe the behaviorof the rest of Kronecker
involutions, namley the ones under conditions $n\equiv0\pmod4$ and $n\mid \frac{k^{2}-1}{2}$, while  $k< \frac{n}{2}$.
In the next subsection we describe their equivalence (for fixed $n,k$), and also
the  corresponding quotient structure.

\subsection{The rim-switching Kronecker involutions}

Let us now turn to the Kronecker involutions containing permutation $\alpha$,
which are described by
an item \ref{thm1:vt} in Theorem~\ref{thm1}, so we assume that
$k^2\equiv 1 \pmod n$ and $Q=\frac{k^2-1}{n}$ is even.

In such case, using  Propositions~\ref{prop:sufficient} and
\ref{_prop:sufficient}
one may find
a Kronecker involution of $\gp{n,k}$, depending on
whether $k\equiv 1 \pmod 4$ or $k\equiv 3 \pmod 4$, respectively. In order to prove that several instances of Kronecker involutions are equivalent, we will need the following extension of the LCF notion.

\begin{defn}
For an involution $g$ without fixed points of
type $[n]\rightarrow [n]$, we define
$f(i)=g(i)-i$ and write, for short, $[f]$ instead
of $[f(0),f(1),\dots,f(n-1)]$.
\end{defn}
It is easy to see that both graphs $C^+(n,k)$ and $C^-(n,k)$ from Definition~\ref{def:cycle+matching} correspond to $[\frac{n}{2}+(k-1)x]$ and $[\frac{n}{2}-(k+1)x]$, respectively.
In order to complete the proof of the main theorem,
it remains to show that for all
possible Kronecker involutions, the corresponding quotient is unique.
%To this end, we first state the following lemma.
We split the further analysis into two cases,
depending the value of $k \pmod 4$.

\subsubsection*{Case 1: $k\equiv 1\pmod 4$}

In this case, any odd $s$ defines $a=sa_{\min}$ and
subseqently a Kronecker
involution of form $\omega_a=\alpha^a\gamma$,
with $\Omega_a(i)=a+ki$.
By Definition~\ref{def:cycle+matching} and
Theorem~\ref{prop:involutions} the corresponding quotient
graph $G'$ is isomorphic to an outer-rim, augmented by
a matching edges of type $i\sim\Omega^{-1}_a(i)$,
which implies $G'\simeq [f_a]$, where
$$f_a(i)=\Omega^{-1}_a(i)-i=ik+a-i.$$

To show that for any odd $s$, all instances of corresponding
Kronecker involutions are equivalent,
we first prove the following lemma.

\begin{lem}\label{lem:equivalent-involutions}
Let $a'=a+\gcd(k-1,Q)\cdot a_{\min}$.
Then $[f_{a}] \simeq [f_{a'}]$.
\end{lem}

\begin{proof}
To prove the claim it is enough to observe that
the LCF sequence of graph $[f_{a'}]$ is equivalent to the LCF sequence
of $[f_{a}]$, cyclically shifted by one, i.e. $f_{a'}(i)=f_a(i+1).$
Indeed, by definition we have
\begin{align*}
f_{a'}(i)&=a+\gcd(k-1,Q)a_{\min}+i(k-1)\\
&= a+(i+1)(k-1)=f_a(i+1),
\end{align*}
where the second line follows from (\ref{eq:alt-notion-amin}).

\end{proof}

We are now ready to show the item \ref{thm1:b1} of Theorem~\ref{thm1}.

\begin{prop}
Let $k^2\equiv 1\pmod n$ with $n\mid \frac{k^2-1}{2}$ and
$k\equiv 1\pmod 4$. Then $\gp{n,k}$ have unique quotient
$[f]$.
\end{prop}

\begin{proof}[Proof of \ref{thm1:b1} of Theorem~\ref{thm1}]
Let $S$ be the set of all $\frac{\gcd(n,k+1)}{2}$
Kronecker involutions.
The Lemma~\ref{lem:equivalent-involutions} partitions $S$
into equivalence classes with respect to the relation of
having an isomorphic corresponding quotients. We will show
that $S$ consists of only one such equivalence class.

In other words, this is equivalent to being in an additive group of order $\frac{\gcd(n,k+1)}{2}$ and calculating an order of element $\frac{\gcd(Q,k-1)}{2}$. Clearly, all classes of such partition of $S$ are of the same cardinality, while the number of these classes is equal to

\begin{align*}
\gcd\left (\frac{\gcd(n,k+1)}{2}, \frac{\gcd(Q,k-1)}{2}\right )&=
\gcd\left (\frac{n}{2},\frac{k+1}{2},\frac{k-1}{2},\frac{Q}{2}\right )\\
&\leq \gcd\left (\frac{k+1}{2},\frac{k-1}{2}\right )=1.
\end{align*}
But then any Kronecker involution, in particular the one with $a=\frac{n}{2}$, corresponds to the unique quotient of $\gp{n,k}$.
\end{proof}

\subsubsection*{Case 2: $k\equiv 3\pmod 4$}

Also in this case, any odd $s$ defines $a=sa'_{\min}$ and
subseqently a Kronecker
involution of form $\omega'_a=\alpha^a\beta\gamma$,
with $\Omega'_a(i)=a-ki$.
By Definition~\ref{def:cycle+matching} and
Theorem~\ref{prop:involutions} the corresponding quotient
graph $G'$ is isomorphic to an outer-rim, augmented by
a matching edges of type $i\sim\Omega^{-1}_a(i)$,
which implies $G'\simeq [f_a]$, where
$$f_a(i)=\Omega'^{-1}_a(i)-i=a-ik-i.$$
We introduce similar lemma as in previous case.

\begin{lem}\label{_lem:equivalent-involutions}
Let $a'=a+\gcd(k+1,k+1-Q)\cdot a'_{\min}$.
Then $[f_{a}] \simeq [f_{a'}]$.
\end{lem}

\begin{proof}
We similarly prove the claim by observing $f_{a'}(i)=f_a(i-1)$.
Indeed, we have
\begin{align*}
f_{a'}(i)&=a+\gcd(k+1,k+1-Q)a'_{\min}-i(k+1)\\
&= a-(i-1)(k+1)=f_a(i-1),
\end{align*}
where the second line follows from (\ref{_eq:alt-notion-amin}).

\end{proof}

We are now ready to show the item \ref{thm1:b2} of Theorem~\ref{thm1}.

\begin{prop}
Let $k^2\equiv 1\pmod n$ with $n\mid \frac{k^2-1}{2}$ and
$k\equiv -1\pmod 4$. Then $\gp{n,k}$ have unique quotient
$C^-(n,k)$.
\end{prop}

\begin{proof}[Proof of \ref{thm1:b2} of Theorem~\ref{thm1}]
Let $S$ be the set of all $\frac{\gcd(n,k+1)}{2}$
Kronecker involutions.
Again, we show that the Lemma~\ref{_lem:equivalent-involutions} eventually covers the whole set $S$.

In this case one may consider an additive group of order $\frac{\gcd(n,n-k+1)}{2}$ and calculate an order of element $\frac{\gcd(k+1,k+1-Q)}{2}$. The number of
such orbits is equal to

\begin{align*}
\gcd\left (\frac{\gcd(n,n-k+1)}{2}, \frac{\gcd(k+1,k+1-Q)}{2}\right )&=
\gcd\left (\frac{n}{2},\frac{k+1}{2},\frac{k-1}{2},\frac{Q}{2}\right )\\
&\leq \gcd\left (\frac{k+1}{2},\frac{k-1}{2}\right )=1.
\end{align*}

\end{proof}
Since we described the quotients of all existing Kronecker involutions,
this concludes the proof of Theorem~\ref{thm1}.

\section{Concluding remarks and future work}

In this paper, we classified parameters $(n,k)$ such that $\gp{n,k}$ is a Kronecker cover of some graph, and described the corresponding quotients.
From the Main Theorem it easily follows:
\begin{cor}
$\kc{\gp{n,k}}$ is itself a generalized Petersen graph if and only if $n$ is odd.
\end{cor}

We analyzed the problem of Kronecker covers of a family of generalized
Petersen graphs. It would be interesting to transfer this problem to the family of $I$-graphs
\citep{boben2005graphs,bouwer1988foster,petkovvsek2009enumeration,horvat2012isomorphism}
or Rose-Window graphs \citep{wilson2008rose},  or some other
families of cubic or quartic graphs.

Graphs $\kc{\gp{n,k}}$ that are not generalized Petersen graphs, in other words if $n$ is odd, fall into
two known classes, depending on the parity of $k$.
If $k$ is odd, we have $\kc{\gp{n,k}} = 2\gp{n,k}$.
It would be interesting to investigate the family of graphs $\kc{\gp{n,k}}$
with both $n$ and $k$ even. The smallest case
is depicted in Figure~\ref{fig:durer}. This
is the Kronecker cover of D\"{u}rer graph $\gp{6,2}$.

\begin{figure}
\centering
\includegraphics[width=7cm]{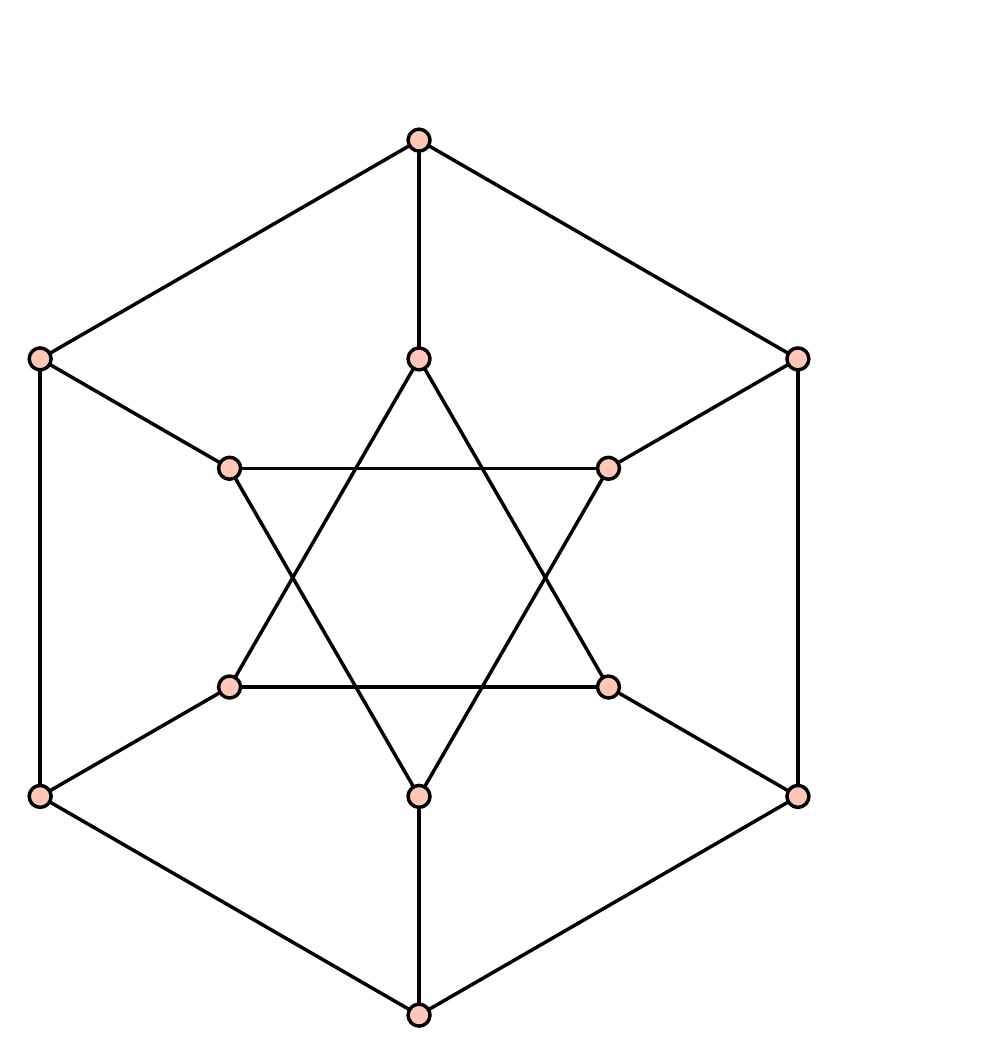}
\quad
\quad
\includegraphics[width=7cm]{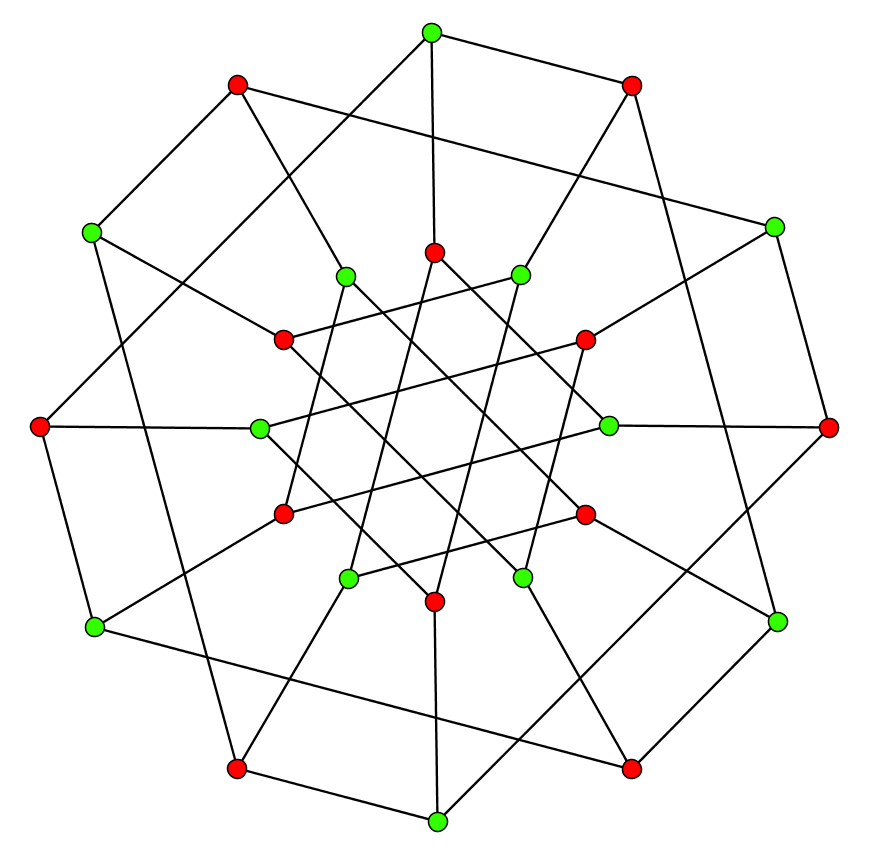}
\caption{\label{fig:durer}The D\"{u}rer graph $\gp{6,2}$ and its Kronecker cover $KC(\gp{6,2})$ with proper vertex two-coloring.}
\end{figure}

\section*{Acknowledgements}
We would like to thank A. Hujdurovi\' c and A. \v Zitnik for helpful comments and discussions and G\'abor G\'evay for the drawing of $KC(\gp{6,2})$.
The second author gratefully acknowledges financial support of the Slovenian Research Agency, ARRS, research program no. P1-0294.

\bibliographystyle{agsm}
\bibliography{kronecker}

\end{document}